\documentclass[12pt]{article}
\usepackage[T2A]{fontenc}
\usepackage[russian,english]{babel}
\usepackage[a4paper,mag=1000,includefoot,left=2cm,right=2cm,top=2cm,bottom=2cm,headsep=1cm,footskip=1cm]{geometry}
\usepackage{amsfonts,amssymb,amsmath,amsthm,indentfirst}
\usepackage{bm}
\usepackage{cite}
\newtheorem{theorem}{\indent Theorem}
\newtheorem{lemma}{\indent Lemma}
\newtheorem{remark}{\indent Remark}
\newtheorem{corollary}{\indent Corollary}
\def\ve{\varepsilon}
\def\vk{\varkappa}
\def\vp{\varphi}
\def\leq{\leqslant}
\def\geq{\geqslant}

\def\*#1{\mathbf{#1}}

\begin{document}
\centerline{\large\textbf{Upper and Lower Error Bounds}} 
\centerline{\large\textbf{for a Compact Fourth-Order Finite-Difference Scheme}}
\centerline{\large\textbf{for the 
Wave Equation with Nonsmooth Data}
}
\bigskip\centerline{A. Zlotnik}
\smallskip
\normalsize
\centerline{\textit{
Higher School of Economics University, Moscow, 109028 Pokrovskii bd. 11 Russia}}
\smallskip
\centerline{\it e-mail: azlotnik@hse.ru}
\smallskip

\begin{abstract}
\noindent A compact three-level fourth-order finite-difference scheme for solving the 1d wave equation is studied.
New error bounds of the fractional order $\mathcal{O}(h^{4(\lambda-1)/5})$ are proved in the mesh energy norm in terms of data, for two initial functions from the Sobolev and Nikolskii spaces with the smoothness orders $\lambda$ and $\lambda-1$ and the free term with a dominated mixed smoothness of order $\lambda-1$, for  $1\leq\lambda\leq 6$. 
The corresponding lower error bounds are proved as well to ensure the sharpness in order of the above error bounds with respect to each of the initial functions and the free term for any $\lambda$. Moreover, they demonstrate that the upper error bounds cannot be improved if the Lebesgue summability
indices in the error norm are weakened down to 1 both in $x$ and $t$ and simultaneously
the summability indices in the norms of data are strengthened up to $\infty$ both in $x$ and $t$. Numerical experiments confirming the sharpness of the mentioned orders for half-integer $\lambda$ and piecewise polynomial data have already been carried out previously. 

\end{abstract}

\section{\normalsize Introduction}

Compact finite-difference schemes of the fourth approximation order ${\mathcal O}(h^4+\tau^4)$ combine high accuracy of results with simplicity of implementation and therefore serve as an important tool to solve numerically various partial differential equations. 
Extensive literature is devoted to the construction, stability analysis, fourth-order error bounds for smooth data and solutions and computer application of such schemes. In the case of the wave equation and its linear  generalizations, in particular, see \cite{V73}-\cite{ZL25} (this list is far from being complete) and a lot of references therein.
However, theoretical analysis of the error of these schemes in the case of non-smooth initial data and the free term of the wave-type equations is complicated and almost absent, although in this case, the fourth order schemes also have confirmed advantages in the practical error behavior over the second order schemes.
 
\par In this paper, we study in detail the error of the known implicit three-level in time compact fourth-order finite-difference scheme  for the initial-boundary value problem for the 1d wave equation.
The essential point is the application of the two-level implicit approximation of the second initial condition $\partial_tu|_{t=0}=u_1$,  without the derivatives of the data of the problem, in contrast to known explicit 5th order approximation of the solution at the first time level involving such derivatives.
New error bounds of the fractional order $\mathcal{O}(h^{4(\lambda-1)/5})$, with $h$ being the mesh step in $x$, are proved in the uniform in time mesh energy norm in terms of the data, rather than under assumptions on the exact solution, for the two initial functions $u_0$ and $u_1$ from the Sobolev spaces (for integer $\lambda$) and Nikolskii spaces (for noninteger $\lambda$) with the smoothness orders $\lambda$ and $\lambda-1$ and the free term $f$ with a dominated mixed smoothness of order $\lambda-1$, for  $1\leq\lambda\leq 6$.
In particular, the difficult case of the piecewise smooth $u_0$ and discontinuous $u_1$ is covered for $\lambda=\frac32$, and the case of the piecewise smooth $\partial_xu_0$ and piecewise smooth $u_1$ is covered for $\lambda=\frac52$.
Importantly, for any $\lambda>1$, the above error bounds are better in order than the corresponding error bounds $\mathcal{O}((h+\tau)^{2(\lambda-1)/3})$, $1\leq\lambda\leq 4$, for the second order finite element method (FEM) with a weight in \cite{Z94}.
Such an effect is absent for the elliptic and parabolic equations, in particular, see \cite{BS08}, but is known for the time-dependent Schr\"{o}dinger equation \cite{BTW75}.
Recall that the cases of discontinuous or nonsmooth data are of interest, in particular, for the corresponding optimal control problems, for example, see \cite{TVZ18,EVT21}.

\par The proof is based on the interpolation of the error operator using the uniform in time energy bounds for the exact and approximate solutions using the weakest data norms and the error bound $\mathcal{O}(h^4)$ under the minimal Sobolev regularity of the exact solution followed by its expression in terms of the data. 
In both the results for the compact scheme, its close connection to the bilinear FEM with a special weight plays an important role.
The proof defers essentially from the corresponding studies of the error of finite-difference schemes in \cite{JS14} and lead to the better results, although in the special case $u_1=0$ and $f=0$, the similar error bound in a weaker mesh norm and under stronger assumptions on $u_0$ was proved in \cite{J94}. 

\par The corresponding lower error bounds for the difference quotient in $x$ of the error and the error itself are derived as well ensuring the sharpness in order of the above error bounds with respect to each of the initial functions and the free term for each $\lambda$. 
Moreover, they demonstrate that the upper error bounds cannot be improved if the Lebesgue summability indices in the error norm are weakened down to 1 both in $x$ and $t$ and simultaneously the summability indices in the norms of the data are strengthened up to $\infty$ both in $x$ and $t$.

\par The proof of the lower error bounds exploits the explicit Fourier-type formulas for the exact and approximate solutions for the harmonic initial data $u_0$ and $u_1$ and the harmonic both in $x$ and $t$ free term $f$ accompanied by a careful choice of the harmonic frequencies in dependence with $h$. 
These lower bounds go back to the lower error bounds for the second order FEM in \cite[Theorem 2]{Z79} whose proof was given in \cite{Z79diss}; also the modified versions of these bounds have recently been presented in \cite{Z25}.  
Later the lower error bounds for the error of the semi-discrete (continuous in time) second and higher order FEMs were given in \cite{R85,H96}.

\par Notice that numerical experiments confirming the sharpness of the proved error bound for half-integer $\lambda$ and piecewise polynomial data have already been carried out previously in \cite{ZK21}.

\par The paper is organized as follows. 
In Section \ref{sec: ibvp and scheme}, we state the initial-boundary value problem for the 1d wave equation and the compact
fourth-order scheme to solve it. We also prove for the scheme the stability theorem in the mesh energy norm in the form required in this paper.
In Section \ref{sec:4th order error}, we derive the 4th order error bound under the minimal 6th order Sobolev regularity assumptions on the exact solution. 
In Section \ref{sec:fract order bounds}, we obtain fractional-order error bounds in terms of data using the results of two previous sections and the interpolation of the error operator.
The last Section \ref{sec:lower bounds} is devoted to the derivation of the corresponding lower error bounds based on the Fourier representations of the exact and approximate solutions.

\section{\normalsize An initial-boundary value problem for the 1d wave equation and the compact fourth-order scheme to solve it}
\label{sec: ibvp and scheme} 
\setcounter{equation}{0}
\setcounter{lemma}{0}
\setcounter{theorem}{0}
\setcounter{remark}{0}
\setcounter{corollary}{0}

We deal with the initial-boundary value problem (IBVP) for the 1d wave equation
\begin{gather}
 \partial_t^2u-a^2\partial_x^2u=f(x,t),\ \ (x,t)\in Q=Q_T=I_X\times I_T,
\label{eq}\\[1mm]
u|_{x=0,X}=0,\ \ u|_{t=0}=u_0(x),\ \ 
\partial_tu|_{t=0}=u_1(x),\ \  x\in I_X
\label{bq}
\end{gather}
under the homogeneous Dirichlet boundary condition.
Here $a=\rm{const}>0$, $I=I_X=(0,X)$ and $I_T=(0,T)$.
Define the collection of data $\*d=(u_0,u_1,f)$.

\par For $u_0\in H_0^1(I)$, $u_1\in L^2(I)$ and $f\in L^{2,1}(Q)$, this problem has a unique {\it weak solution from the energy class (space)} having the properties $u\in C(\bar{I}_T;H_0^1(I))$ and $\partial_tu\in C(\bar{I}_T;L^2(I))$, and the energy bound holds
\begin{gather}
\max_{0\leq t\leq T}\big(\|\partial_tu(\cdot,t)\|_{L^2(I)}^2+a^2\|u(\cdot,t)\|_{H_0^1(I)}^2\big)^{1/2}
\leq\big(a^2\|u_0\|_{H_0^1(I)}^2+\|u_1\|_{L^2(I)}^2\big)^{1/2}+2\|f\|_{L^{2,1}(Q)},
\label{en bound u}
\end{gather}
see, for example, \cite{L73,DL00} (with an addition concerning $f$ in \cite{Z94}). 
Hereafter $H_0^1(I)=\{w\in W^{1,2}(I);\,w|_{x=0,X}=0\}$ is a Sobolev subspace, $L^{2,q}(Q)$ is the anisotropic Lebesgue space and $C(\bar{I}_T;H)$ is the space of continuous functions on $\bar{I}_T$ with the values in a Hilbert space $H$, endowed with the norms
\[
\|w\|_{H_0^1(I)}=\|\partial_xw\|_{L^2(I)},\ \|f\|_{L^{2,q}(Q)}=\|\|f\|_{L^2(I)}\|_{L^q(I_T)},\, 1\leq q\leq\infty,\
\|u\|_{C(\bar{I}_T;\,H)}=\max_{0\leq t\leq T}\|u(\cdot,t)\|_H.
\]
Since $C(\bar{I}_T;H_0^1(I))\subset C(\bar{Q})$, such solutions $u$ are continuous in $\bar{Q}$.

\par Let $\bar{\omega}_h$ and $\bar{\omega}^\tau$ be the uniform meshes with the nodes $x_i=ih$, $0\leq i\leq N$, in $\bar{I}_X$ and $t_m=m\tau$, $0\leq m\leq M$, in $\bar{I}_T$, with the steps $h=X/N$ and $\tau=T/M$.
Let $\omega_h=\bar{\omega}_h\setminus\{0,X\}$, $\omega^\tau=\bar{\omega}^\tau\setminus\{0,T\}$ and 
$\bar{\omega}_{\*h}=\bar{\omega}_h\times \bar{\omega}^\tau$ with $\*h:=(h,\tau)$.
Let $w_i=w(x_i)$, $y^m=y(t_m)$ and  $v_i^m=v(x_i,t_m)$.

\par We introduce the mesh averages and finite-difference operators in $x$ and $t$:
\begin{gather*}
s_Nw_i=\frac{1}{12}(w_{i-1}+10w_i+w_{i+1}),\ \ 
Bw_i=\frac{1}{6}(w_{i-1}+4w_i+w_{i+1}),\ \ 
\Lambda_xw_i=\frac{w_{i-1}-2w_i+w_{i+1}}{h^2},
\\[1mm]
\bar{s}_ty=\frac{y+\check{y}}{2},\ \
s_{Nt}y=\frac{1}{12}(\hat{y}+10y+\check{y}),\ \
\delta_t y=\frac{\hat{y}-y}{\tau},\ \
\bar{\delta}_ty=\frac{y-\check{y}}{\tau},\ \ \Lambda_ty=\frac{\hat{y}-2y+\check{y}}{\tau^2},
\end{gather*}
where $\check{y}^m=y^{m-1}$, $\hat{y}^m=y^{m+1}$.
The formulas hold 
\[
s_N=\mathbf{I}+\tfrac{1}{12}h^2\Lambda_x,\ \ 
B=\mathbf{I}+\tfrac{1}{6}h^2\Lambda_x,\ \
s_{Nt}=\mathbf{I}+\tfrac{1}{12}\tau^2\Lambda_t,\ \
\Lambda_t=\delta_t\bar{\delta}_t, 
\]
where $\mathbf{I}$ is the identity operator.
Here $s_N$ and $s_{Nt}$ are the Numerov averaging operators in $x$ and $t$, and $B$ is the averaging operator (the scaled mass matrix) corresponding to the linear finite elements. 

\par We introduce the Euclidean space $H_h$ of functions given on $\bar{\omega}_h$ and equal 0 at $x_i=0,X$, endowed with the inner product 
\[
(v,w)_h=\sum_{i=1}^{N-1}v_iw_ih.
\]
For a linear operator $A=A^*>0$ acting in $H_h$, we also use the norm $\|w\|_{A}=(Aw,w)_h^{1/2}$ in $H_h$. 
Let $\|A\|$ be its norm. 
The operators $s_N$, $B$ and $\Lambda_x$ can be considered as acting in $H_h$ after setting $(s_Nw)|_{i=0,N}=(Bw)|_{i=0,N}=(\Lambda_xw)|_{i=0,N}=0$.
It is well-known that then
\begin{gather}
\tfrac23\mathbf{I}\leq s_N=s_N^*\leq \mathbf{I},\ \ \tfrac13\mathbf{I}\leq B=B^*\leq \mathbf{I}, \  0<-\Lambda_x=-\Lambda_x^*\leq\tfrac{4}{h^2}\mathbf{I},
\label{oper ineq}\\
\|w\|_{-\Lambda_x}^2=(-\Lambda_xw,w)
=\|\bar{\delta}_xw\|_{h*}^2:=\sum_{i=1}^N(\bar{\delta}_xw_i)^2h\,\ \text{for}\,\ w\in H_h,
\ \ \text{with}\ \ \bar{\delta}_xw_i=\frac{w_i-w_{i-1}}{h}.
\nonumber
\end{gather}

\par To solve the IBVP \eqref{eq}--\eqref{bq}, we consider the known three-level fourth-order compact scheme 
\begin{gather}
\big(s_N-\tfrac{1}{12}\tau^2a^2\Lambda_x\big)\Lambda_t v-a^2\Lambda_xv=f_{\*h}\ \ \text{on}\ \ \omega_h\times\omega^\tau,
\label{eqht}\\
\big(s_N-\tfrac{1}{12}\tau^2a^2\Lambda_x\big)\delta_t v^0-\tfrac{\tau}{2}a^2\Lambda_xv^0=u_{1\*h}+\tfrac{\tau}{2}f_{\*h}^0\ \ \text{on}\ \ \omega_h,
\label{inht1}\\
 v^0=u_0\ \ \text{on}\ \ \bar{\omega}_h,\ \
v_i^m|_{i=0,N}=0,\ \
1\leq m\leq M,
\label{bqht}
\end{gather}  
in particular, see \cite{V73,ZK21}.
The usage of the two-level implicit initial condition  \eqref{inht1} for equation  \eqref{eqht} was suggested in \cite{ZK21}, and this is essential for non-smooth data since it does not involve derivatives of the data in contrast to an alternative explicit formula for $v^m|_{m=1}$ having the 5th truncation order.
The scheme has the 4th truncation order for smooth solutions $u$ under the proper choice of $f_{\*h}$, $f_{\*h}^0$ and $u_{1\*h}$.
The standard choice of $f_{\*h}$ is $f_{\*h}=f+\tfrac{h^2}{12}\Lambda_xf+\tfrac{\tau^2}{12}\Lambda_tf$.
It follows from \cite{ZK21} that, for example, one can also set
\begin{gather*}
f_{\*h}=q_hq_\tau f\ \ \text{on}\ \ \omega_h\times\omega^\tau,
\\
u_{1\*h}=u_{1\*h}^{(0)}:=s_Nu_1+\tfrac{\tau^2}{12}a^2\Lambda_xu_1\ \ \text{or}\ \ 
u_{1\*h}= u_{1\*h}^{(1)}:=q_hu_1+\tfrac{\tau^2}{12}a^2\Lambda_xu_1;
\end{gather*}  
the form of $f_{\*h}^0$ from \cite{ZK21} is omitted.
The choice $f_{\*h}^m=(q_hq_\tau f)^m$,  $0\leq m\leq M-1$, is suitable for $f\in L^{2,1}(Q)$, and we use namely it below. 
But both the formulas for $u_{1\*h}$ are not suitable for $u_1\in L^2(I)$, and therefore another formula is suggested and applied below. 
\par Hereafter, for $w\in L^1(I)$ and $z\in L^1(I_T)$ the following averages are applied
\begin{gather*}
 q_{0h}w_i=\frac{1}{2h}\int_{x_{i-1}}^{x_{i+1}}w(x)\,dx,\ \
 q_{0\tau}z^0=\frac{1}{\tau}\int_0^\tau z(t)\,dt,\ \
 q_{0\tau}z^m=\frac{1}{2\tau}\int_{t_{m-1}}^{t_{m+1}} z(t)\,dt,
\\
q_hw_i=\frac{1}{h}\int_{x_{i-1}}^{x_{i+1}}w(x)e_i^h(x)\,dx,\ \
 q_\tau z^0=\frac{2}{\tau}\int_0^\tau z(t)e^{\tau,0}(t)\,dt,\ \
q_\tau z^m=\frac{1}{\tau}\int_{t_{m-1}}^{t_{m+1}} z(t)e^{\tau,m}(t)\,dt,
\label{avertau}
\end{gather*}
where $1\leq i\leq N-1$, $1\leq m\leq M-1$, and the well-known ``hat'' functions 
\[
e_i^h(x)=\max\big\{1-\big|\tfrac{x}{h}-i\big|,0\big\},\ \ 
e^{\tau,m}(t)=\max\big\{1-\big|\tfrac{t}{\tau}-m\big|,0\big\}
\]
from bases in the spaces of piecewise linear finite elements in $x$ and $t$ are applied.

\par We pass to a theorem on conditional stability of the introduced compact scheme. The known stability results, including that from \cite{ZK21}, are not fully suitable for our purposes, and we need another one. Define the quantity
\[
\|\{\check{v},v\}\|_{E_{\*h}}=\big(\|\bar{\delta}_tv\|_{B}^2
+(\sigma_N-\tfrac14)a^2\|\bar{\delta}_tv\|_{-\Lambda_x}^2
+a^2\|\bar{s}_tv\|_{-\Lambda_x}^2\big)^{1/2},\ \ \text{with}\ \ \sigma_N:=\tfrac{1}{12}\big(1+\tfrac{h^2}{a^2\tau^2}\big).
\]
It serves as the level energy norm for this scheme and is well defined under the stability condition
\begin{gather}
a^2\tau^2\leq (1-\tfrac{\ve_0^2}{2})h^2 
\label{stab cond 1}
\end{gather}
with some  $0<\ve_0\leq 1$. As it is shown below, the condition guarantees the lower bound 
\begin{gather}
\ve_0^2\|\bar{\delta}_tv\|_B^2+a^2\|\bar{s}_tv\|_{-\Lambda_x}^2\leq\|\{\check{v},v\}\|_{E_{\*h}}^2\ \ \text{for any}\ \ \check{v},v\in H_h.  
\label{lower bound en norm 1} 
\end{gather}

We also introduce the norm $\|F\|_{L_{\*h}^{2,1}}=\tau\sum_{m=1}^{M-1}\|F^m\|_h$ for functions $F$ given on $\omega_h\times\omega^\tau$.
\begin{theorem}
\label{theo:stability}
Under the stability condition \eqref{stab cond 1}, for the compact scheme \eqref{eqht}--\eqref{bqht}, the energy bound holds
\begin{gather}
\max_{1\leq m\leq M}\|\{\check{v},v\}^m\|_{E_{\*h}}
\nonumber\\
\leq\big(a^2\|(-\Lambda_x)^{1/2}v^0\|_h^2+\ve_0^{-2}\|B^{-1/2}u_{1\*h}\|_h^2\big)^{1/2}
+\ve_0^{-1}\big(\|B^{-1/2}f_{\*h}^0\|_h\tau+2\|B^{-1/2}f_{\*h}\|_{L_{\*h}^{2,1}}\big).
\label{stab bound}
\end{gather}
Here the initial data $v^0$ and $u_{1\*h}$ and the free term $f_{\*h}$ are arbitrary.
\end{theorem}
\begin{proof}
We rewrite equations \eqref{eqht}--\eqref{inht1} in virtue of the formula $s_N=B-\frac{h^2}{12}\Lambda_x$ in the form
\begin{gather}
\big(B-\sigma_N\tau^2a^2\Lambda_x\big)\Lambda_t v-a^2\Lambda_xv=f_{\*h}\ \ \text{on}\ \ \omega_h\times\omega^\tau,
\label{eqht fem}\\
\big(B-\sigma_N\tau^2a^2\Lambda_x\big)\delta_t v^0-\tfrac{\tau}{2}a^2\Lambda_xv^0=u_{1\*h}+\tfrac{\tau}{2}f_{\*h}^0\ \ \text{on}\ \ \omega_h.
\label{inht1 fem}
\end{gather}
Importantly, these equations for $f_{\*h}=q_hq_\tau f$ and $u_{1\*h}=q_hu_1$ coincide with the equations of the bilinear FEM with the special weight $\sigma=\sigma_N$, see  \cite{Z79,Z94}, which was indicated in  \cite{ZK21}.

\par Let $a^2(-\Lambda_x)\leq\alpha_h^2B$ in $H_h$.
According to \cite[Section 2.2]{Z94}, under the condition 
\begin{gather}
 \sigma_N\geq\tfrac14-\tfrac{1-\ve_0^2}{\tau^2\alpha_h^2}\ \ \text{with}\ \ 0<\ve_0\leq 1,    
\label{stab cond fem}
\end{gather}
the following inequality holds
\[
\ve_0^2\|w\|_B^2\leq \|w\|_B^2+\big(\sigma_N-\tfrac14\big)\tau^2a^2\|w\|_{-\Lambda_x}^2\ \ \text{for}\ \ w\in H_h.
\]
In virtue of the operator inequalities \eqref{oper ineq}, one can set $\alpha_h^2=\frac{12a^2}{h^2}$, then condition \eqref{stab cond fem} becomes equivalent to \eqref{stab cond 1}, and the last inequality implies \eqref{lower bound en norm 1}.

\par Now bound \eqref{stab bound} follows from the  stability result for an abstract three-level method with a weight \cite[Theorem 1]{ZK21}. 
\end{proof}

Recall that bounds \eqref{lower bound en norm 1} and \eqref{stab bound} ensure existence and uniqueness of the approximate solution $v$.
\begin{corollary}
\label{corr: v en bound}
Let $u_0\in H_0^1(I)$, $u_1\in L^2(I)$ and $f\in L^{2,1}(Q)$.
Under the stability condition \eqref{stab cond 1} with some  $0<\ve_0^2\leq 1$,
bound \eqref{stab bound} for $v^0=u_0$, $u_{1\*h}=u_{1\*h}^{(2)} :=\big(\mathbf{I}+\tfrac{\tau^2}{12}a^2\Lambda_x\big)q_hu_1$ with $q_hu_1|_{i=0,N}:=0$ and $f_{\*h}=q_hq_\tau f$ implies the bound
\begin{gather}
\ve_0\max\Big\{\max_{1\leq m\leq M}\|\bar{\delta}_tv^m\|_B,
\max_{0\leq m\leq M}\tfrac{1}{\sqrt{6}}a\|\bar{\delta}_xv^m\|_{h*}
\Big\}
\nonumber\\  
\leq\big(a^2\|u_0\|_{H_0^1(I)}^2+\ve_0^{-2}\|u_1\|_{L^2(I)}^2\big)^{1/2}+2\ve_0^{-1}\|f\|_{L^{2,1}(Q)}.
\label{mesh energy bound}
\end{gather}
\end{corollary}
\begin{proof}
According to \cite[Lemma 2.1]{Z94}, under the condition
\[ 
\sigma_N\geq\tfrac{1+\ve_0^2}{4}-\tfrac{1}{\tau^2\alpha_h^2}\ \ \text{with}\ \ 0<\ve_1\leq 1, 
\]
the following inequality holds
\begin{gather}
\ve_1^2a^2\bar{s}_t(\|v\|_{-\Lambda_x}^2)\leq\|\{\check{v},v\}\|_{E_{\*h}}^2\ \ \text{for any}\ \ \check{v},v\in H_h.
\label{lower bound en norm 2} 
\end{gather}
This condition on $\sigma_N$ for $\alpha_h^2=\frac{12a^2}{h^2}$ 
follows from the above condition \eqref{stab cond 1} provided that $\ve_1^2\leq\ve_0^2\frac{h^2}{3a^2\tau^2}$.
Thus, it is enough to assume that $\ve_1^2\leq \ve_0^2/[3(1-\frac{\ve_0^2}{2})]$ or simply to choose $\ve_1^2=\ve_0^2/3$.
Therefore, the left-hand side of the energy bound \eqref{stab bound} is estimated from below by the left-hand side of bound \eqref{mesh energy bound}.

\par Further, first, we have
\begin{equation}
\|(-\Lambda_x)^{1/2}u_0\|_h=\|\bar{\delta}_xu_0\|_{h*}
=\|q_{0h}^{(1)}\partial_xu_0\|_{h*}
\leq\|\partial_xu_0\|_{L^2(I)},
\label{bound of fd of u}
\end{equation}
with $q_{0h}^{(1)}w_i:=\frac1h\int_{x_{i-1}}^{x_i}w(x)\,dx$, $1\leq i\leq N$.
Second, under the stability condition \eqref{stab cond 1}, we have
$0<\mathbf{I}+\tfrac{\tau^2}{12}a^2\Lambda_x<\mathbf{I}$ in  $H_h$ and thus $\|\mathbf{I}+\tfrac{\tau^2}{12}a^2\Lambda_x\|\leq 1$.
Therefore 
\[
\|B^{-1/2}u_{1\*h}^{(2)}\|_h\leq \|B^{-1/2}q_hu_1\|_h\leq\|u_1\|_{L^2(I)},\]
where the latter inequality follows from \cite{Z94}.
Similarly, the following inequalities hold
\begin{gather*}
\|B^{-1/2}q_hq_\tau f^0\|_h\tau+2\|B^{-1/2}q_hq_\tau f\|_{L_{\*h}^{2,1}}
\leq \tau\|q_\tau f^0\|_{L^2(I)}+2\tau\sum_{m=1}^{M-1}\|q_\tau f^m\|_{L^2(I)}
\\
\leq\tau (q_\tau\|f(\cdot,t)\|_{L^2(I)})^0+2\tau\sum_{m=1}^{M-1}(q_\tau\|f(\cdot,t)\|_{L^2(I)})^m
\leq 2\|f\|_{L^{2,1}(Q)},
\end{gather*}
where the Minkowski generalized integral inequality is also applied. 
Consequently, for the chosen $v^0$, $u_{1\*h}$ and $f_{\*h}$, the right-hand side of the energy bound \eqref{stab bound} is majorized by the right-hand side of bound \eqref{mesh energy bound}.
\end{proof}

\section{\normalsize The 4th order error bound}
\label{sec:4th order error}
\setcounter{equation}{0}
\setcounter{lemma}{0}
\setcounter{theorem}{0}
\setcounter{remark}{0}
\setcounter{corollary}{0}

In this Section, we consider solutions to the IBVP \eqref{eq}--\eqref{bq} having the additional regularity properties $\partial_x^2u,\partial_t^2u,\partial_x^4u,\partial_t^4u,\partial_x^4\partial_t^2u,\partial_x^2\partial_t^4u\in L^{2,1}(Q)$. Consequently, $\partial_x^4u_1=\partial_x^4(\partial_tu)_0\in L^2(I)$; hereafter $g_0=g|_{t=0}$ for functions $g$ depending on $t$.
\begin{theorem}
\label{theo:err bound 4th}
Under the stability condition \eqref{stab cond 1}, for the compact scheme \eqref{eqht}--\eqref{bqht} with $u_{1\*h}=u_{1\*h}^{(1)}=q_hu_1+\tfrac{\tau^2}{12}a^2\Lambda_xu_1$ and $f_{\*h}=q_hq_\tau f$, the 4th order error bound holds 
\begin{gather}
\max_{1\leq m\leq M}\|\{\check{u}-\check{v},u-v\}^m\|_{E_{\*h}}
\nonumber\\
\leq ch^4\big(\|\partial_x^4\partial_t^2u\|_{L^{2,1}(Q)}
+\|\partial_x^2\partial_t^4u\|_{L^{2,1}(Q)}
+\|\partial_x^2\partial_t^3u\|_{L^{2,1}(Q)}
+\|\partial_x^4u_1\|_{L^2(I)}
\big).
\label{r bound 0}
\end{gather}
Hereafter constants $c$, $c_1$, $c_2$, etc. can depend on $X,T,\ve_0$ but are independent of $\*h$.
\end{theorem}
\begin{proof}
Applying to the wave equation \eqref{eq} the operator $q_hq_\tau$ leads to the equalities
\begin{gather*}
\Lambda_tq_hu-a^2\Lambda_xq_\tau u=q_hq_\tau f\ \ \text{on}\ \ \omega_h\times\omega^\tau,
\label{eqht u}\\
\delta_tq_hu^0-\tfrac{\tau}{2}a^2\Lambda_xq_\tau u^0=q_hu_1+\tfrac{\tau}{2}q_hq_\tau f^0\ \ \text{on}\ \ \omega_h.
\label{inht1 u}
\end{gather*}
They are mentioned in the proofs of Lemmas 1 and 2 in \cite{ZK21} and follow from the known formulas
\begin{gather}
 \Lambda_xw=q_h(\partial_x^2w),\ \
\Lambda_tz=q_\tau(\partial_t^2z),\ \
\tfrac{\tau}{2}(q_\tau\partial_t^2z)^0=\delta_tz^0-\partial_tz|_{t=0},
\label{form with Lambda} 
\end{gather}
for $w\in W^{2,1}(I)$ and $z\in W^{2,1}(I_T)$.

\par Therefore, in virtue of equations of the compact scheme \eqref{eqht}--\eqref{bqht}, for the error $r:=u-v$, we derive the following equation
\begin{gather*}
\big(s_N-\tfrac{1}{12}\tau^2a^2\Lambda_x\big)\Lambda_t r-a^2\Lambda_xr
=\big(s_N-\tfrac{1}{12}\tau^2a^2\Lambda_x\big)\Lambda_tu-a^2\Lambda_xu-
(\Lambda_tq_hu-a^2\Lambda_xq_\tau u)
\\
=(s_N-q_h)\Lambda_tu-a^2(s_{Nt}-q_\tau)\Lambda_xu
\ \ \text{on}\ \ \omega_h\times\omega^\tau
\end{gather*}
together with
\begin{gather*}
\big(s_N-\tfrac{1}{12}\tau^2a^2\Lambda_x\big)\delta_t r^0-\tfrac{\tau}{2}a^2\Lambda_xr^0
\\[1mm]
=\big(s_N-\tfrac{1}{12}\tau^2a^2\Lambda_x\big)\delta_tu^0-\tfrac{\tau}{2}a^2\Lambda_xu^0-\big(\delta_tq_hu^0-\tfrac{\tau}{2}a^2\Lambda_xq_\tau u^0\big)
+q_hu_1-u_{1\*h}
\\
=(s_N-q_h)\delta_tu^0-a^2\big\{\big[\tfrac{\tau}{2}\mathbf{I}+\tfrac{1}{12}\tau^2(\delta_t+\partial_t)-\tfrac{\tau}{2}q_\tau\big]\Lambda_xu\big\}^0
\ \ \text{on}\ \ \omega_h,
\label{inht1 r}\\[1mm]
 r^0=0\ \ \text{on}\ \ \bar{\omega}_h,\ \ r_i^m|_{i=0,N}=0,\ \
1\leq m\leq M.
\end{gather*}
Here we have used the formula $q_hu_1-u_{1\*h}=-a^2\frac{\tau^2}{12}(\partial_t\Lambda_xu)^0$ for $u_{1\*h}=u_{1\*h}^{(1)}$.
Therefore, due to the stability Theorem
\ref{theo:stability} and the estimate $\|B^{-1/2}\|\leq\sqrt{3}$, the error bound holds 
\begin{gather}
\max_{1\leq m\leq M}\|\{\check{r},r\}^m\|_{E_{\*h}}
\leq \sqrt{3}\ve_0^{-1} \big(2\|(s_N-q_h)\Lambda_tu\|_{L_{\*h}^{2,1}}
+2a^2\|(s_{Nt}-q_\tau)\Lambda_xu\|_{L_{\*h}^{2,1}}
\nonumber\\
+\|(s_N-q_h)\delta_tu^0\|_h
+a^2\big\|\big\{\big[\tfrac{1}{2}\tau \mathbf{I}+\tfrac{1}{12}\tau^2(\delta_t+\partial_t)-\tfrac{1}{2}\tau q_\tau\big]\Lambda_xu\big\}^0\big\|_h
\big).
\label{r bound 1}
\end{gather}

\par Using the Taylor formula with the center at the points $x_i$, $t_m$ and $t_0=0$ and the remainder in the integral form, one can verify the 4th order bounds
\begin{gather}
 |(q_hw-w)_i|\leq ch^2(q_{0h}|\partial_x^2w|)_i,
\label{est qhw min w}\\
 |(q_hw-s_Nw)_i|\leq ch^4(q_{0h}|\partial_x^4w|)_i,
\ \
 |(q_\tau z-s_{Nt}z)^m|\leq c\tau^4(q_{0\tau}|\partial_t^4z|)^m,
\label{est qh min sN}\\
\big|\big\{\big[\tfrac{1}{2}\tau q_\tau-\tfrac{1}{2}\tau \mathbf{I}-\tfrac{1}{12}\tau^2(\delta_t+\partial_t)\big]z\big\}^0\big|\leq c\tau^4(q_{0\tau}|\partial_t^3z|)^0,
\nonumber
\end{gather}
where $1\leq i\leq N-1$ and $1\leq m\leq M-1$, for functions $w\in L^1(I)$ and $z\in L^1(I_T)$ having the derivatives $\partial_x^2w,\partial_x^4w\in L^1(I)$, $\partial_t^4z\in L^1(I_T)$ and $\partial_t^3z\in L^1(I_\tau)$, respectively.
Clearly, the second and third of these bounds are quite similar.
Bounds similar to all the presented ones are contained in the proofs of Lemmas 1 and 2 in \cite{ZK21}.

\par In virtue of the listed bounds and the above formulas \eqref{form with Lambda} for $u$, bound \eqref{r bound 1} implies
\begin{gather*}
\max_{1\leq m\leq M}\|\{\check{r},r\}^m\|_{E_{\*h}}
\leq c\big(h^4\|q_{0h}q_\tau|\partial_x^4\partial_t^2u|\|_{L_{\*h}^{2,1}}  +\tau^4\|q_hq_{0\tau}|\partial_x^2\partial_t^4u|\|_{L_{\*h}^{2,1}}
\nonumber\\[1mm] 
+h^4\|q_{0h}(q_{0\tau}|\partial_x^4\partial_tu|)^0|\|_h
+\tau^4\|q_h(q_{0\tau}|\partial_x^2\partial_t^3u|)^0|\|_h\big).
\end{gather*}
The mesh norms of the averages of functions on the right-hand side of this bound are less or equal to the corresponding norms of the functions themselves, in general, with a multiplier $c$.
Therefore, using also the stability condition  \eqref{stab cond 1}, we obtain  the further bound
\begin{gather*}
\max_{1\leq m\leq M}\|\{\check{r},r\}^m\|_{E_{\*h}}\leq
ch^4\big(\|\partial_x^4\partial_t^2u\|_{L^{2,1}(Q)}
+\|\partial_x^2\partial_t^4u\|_{L^{2,1}(Q)}
\nonumber\\
+\|\partial_x^4\partial_tu\|_{L^{2,\infty}(Q_\tau)}
+\|\partial_x^2\partial_t^3u\|_{L^{2,\infty}(Q_\tau)}
\big).
\end{gather*}
Applying the simple estimates for the 5th order derivatives of $u$ of the form
\begin{gather}
\|\partial_x^4\partial_tu\|_{L^{2,\infty}(Q_\tau)}\leq \|\partial_x^4u_1\|_{L^2(I)}+\|\partial_x^2\partial_t^4u\|_{L^{2,1}(Q)},
\label{px4 pt1 bound}\\
\|\partial_x^2\partial_t^3u\|_{L^{2,\infty}(Q_\tau)}\leq \tfrac1T\|\partial_x^2\partial_t^3u\|_{L^{2,1}(Q)}+\|\partial_x^2\partial_t^4u\|_{L^{2,1}(Q)},
\nonumber
\end{gather}
we complete the proof.    
\end{proof}

The following addition is important below.
\begin{remark}
\label{rem: another u1h}
The error bound \eqref{r bound 0} remains valid for $u_{1\*h}=u_{1\*h}^{(0)}$
and, in the case $u_1|_{x=0,X}=\partial_x^2u_1|_{x=0,X}=0$, also for $u_{1\*h}=u_{1\*h}^{(2)}$.
\end{remark}
\begin{proof}
In virtue of the stability Theorem \ref{theo:stability}, the replacement of $u_{1\*h}=
u_{1\*h}^{(1)}$ with $u_{1\*h}^{(0)}$ or  $u_{1\*h}^{(2)}$ implies the appearance of the additional summands $\sqrt{3}\ve_0^{-1}\|(s_N-q_h)u_1\|_h\leq ch^4\|\partial_xu_1\|_{L^2(I)}$ or
$S:=\sqrt{3}\ve_0^{-1}a^2\tfrac{1}{12}\tau^2\|\Lambda_x(u_1-q_hu_1)\|_h$, respectively, on the right-hand side of bound \eqref{r bound 1}. 

\par Let us bound $S$. 
To this end, we enlarge the above introduced notation for any $x\in \bar{I}_X$:
\[
(\tilde{\Lambda}_xw)(x):=\frac{1}{h^2}(w(x+h)-2w(x)+w(x-h)),\ \ (\tilde{q}_hw)(x):=\frac1h\int_{x-h}^{x+h}w(\xi)\Big(1-\frac{|\xi-x|}{h}\Big)\,d\xi
\]
for $w\in L^1(\tilde{I})$, where $\tilde{I}:=(-X,2X)$. 
Then the following formulas hold
\begin{gather*}
(\Lambda_xq_hw)_i=\frac{1}{h^3}\int_{-h}^{X+h} w(x)(e_{i+1}^h-2e_i^h+e_{i-1}^h)(x)\,dx=\frac1h\int_0^X(\tilde{\Lambda}_xw)(x)e_i^h(x)\,dx=(q_h(\tilde{\Lambda}_xw))_i,
\\
(\tilde{\Lambda}_xw)(x)=\tilde{q}_h(\partial_x^2w)(x),\ \ x\in \bar{I}_X,
\end{gather*}
where $1\leq i\leq N-1$ and, in the latter formula, it is assumed that  $\partial_x^2w\in L^1(\tilde{I})$. 

\par If $w\in W^{4,2}(I)$, $w|_{x=0,X}=\partial_x^2w|_{x=0,X}=0$ and $w$ is extended oddly with respect to $x=0,X$ on $\tilde{I}$, then there exist $\partial_x^2w,\partial_x^4w\in L^2(\tilde{I})$ and $q_hw|_{x=0,X}=0$.
Therefore, the bound holds
\begin{gather}
\|\Lambda_x(u_1-q_hu_1)\|_h=\|\tilde{\Lambda}_xu_1-q_h\tilde{\Lambda}_xu_1\|_h
\leq c_1h^2\|q_{0h}|\partial_x^2\tilde{\Lambda}_xu_1|\|_h
\leq c_2h^2\|\partial_x^4u_1\|_{L^2(I)},    
\label{est lam mim lam q}
\end{gather}
since bound \eqref{est qhw min w} is valid and $\partial_x^2\tilde{\Lambda}_xu_1=\tilde{\Lambda}_x\partial_x^2u_1=\tilde{q}_h(\partial_x^4u_1)$. 
Thereby $S\leq c\tau^2h^2\|\partial_x^4u_1\|_{L^2(I)}$.
\end{proof}

\section{\normalsize Fractional-order error bounds in terms of data}
\label{sec:fract order bounds}
\setcounter{equation}{0}
\setcounter{lemma}{0}
\setcounter{theorem}{0}
\setcounter{remark}{0}
\setcounter{corollary}{0}

We first consider the Fourier series with respect to sines $w(x)=\sum_{k=1}^\infty\tilde{w}_k\sin\frac{\pi kx}{X}$ for functions $w\in L^2(I)$ and define the well-known Hilbert spaces of functions
\[
\mathcal{H}^\alpha(I)=\Big\{w\in L^2(I);\,
\|w\|_{\mathcal{H}^\alpha}^2=\sum_{k=1}^\infty \big(\tfrac{\pi k}{X}\big)^{2\alpha}\tilde{w}_k^2<\infty,\ \ 
\tilde{w}_k=\sqrt{\tfrac2X}\int_0^Xw(x)\sin\tfrac{\pi kx}{X}\,dx\Big\},\ \ \alpha\geq 0.
\]
Here $\mathcal{H}^0(I)\equiv L^2(I)$, and the space $\mathcal{H}^\alpha(I)$ for $\alpha=k\in\mathbb{N}$ coincides with the subspace in the Sobolev space
\[
 \mathcal{H}^k(I)=\big\{w\in W^{k,2}(I);\, \partial_x^{2j}w|_{x=0,X}=0,\, 0\leq 2j<k\big\},\ \ \|w\|_{\mathcal{H}^k(I)}=\|\partial_x^kw\|_{L^2(I)}.
\]
In particular, $\mathcal{H}^1(I)=H_0^1(I)$.

We also define the Banach spaces $S_{2,1}^{3,2}W(Q)$ and $S_{2,1}^{2,3}W(Q)$ of functions $f\in L^{2,1}(Q)$ having the dominating mixed smoothness, endowed with the norms
\begin{gather*} \|f\|_{S_{2,1}^{3,2}W(Q)}=\|\partial_t^2f\|_{L^1(I_T;\mathcal{H}^3(I))}+\|f_0\|_{\mathcal{H}^4(I)}+\|(\partial_tf)_0\|_{\mathcal{H}^3(I)},
\\
\|f\|_{S_{2,1}^{2,3}W(Q)}=\|\partial_t^3f\|_{L^1(I_T;\mathcal{H}^2(I))}+\|f_0\|_{\mathcal{H}^4(I)}+\|(\partial_tf)_0\|_{\mathcal{H}^3(I)}+\|(\partial_t^2f)_0\|_{\mathcal{H}^2(I)};
\end{gather*}
in this respect see, in particular, \cite{N63,S07,T19}.
Here it is assumed that respectively
$\partial_t^\ell f\in L^1(I_T;\mathcal{H}^3(I))$ for $0\leq \ell\leq 2$, $f_0\in \mathcal{H}^4(I)$ or $\partial_t^\ell f\in L^1(I_T;\mathcal{H}^2(I))$ for $0\leq \ell\leq 3$, $f_0\in \mathcal{H}^4(I)$, $(\partial_tf)_0\in \mathcal{H}^3(I)$.
Note that the additional properties respectively $(\partial_tf)_0\in \mathcal{H}^3(I)$ or $(\partial_t^2f)_0\in \mathcal{H}^2(I)$ follow from the listed ones. 
To simplify the notation of these spaces, we do not  mark the arisen particular conditions of taking zero values at $x=0,X$.

\par Let $H^{(k)}=\mathcal{H}^k(I)$ for integer $k\geq 0$, $F^{(0)}=L^{2,1}(Q)$ and  $F^{(5)}=S_{2,1}^{3,2}W(Q)+S_{2,1}^{2,3}W(Q)$ be the sum of (compatible) Banach spaces, for example, see \cite{BL80}. 
We define the Banach spaces of functions having an intermediate smoothness
\[
  H^{(\varkappa)}=(\mathcal{H}^k(I),\mathcal{H}^{k+1}(I))_{\varkappa-k,\infty},\ \ F^{(\lambda)}=(L^{2,1}(Q),S_{2,1}^{3,2}W(Q)+S_{2,1}^{2,3}W(Q))_{\lambda/5,\infty},\ \ 0<\lambda<5,
\]
where $\varkappa>0$ is noninteger and $k$ is the integer part of $\varkappa$.
Also $(B_0,B_1)_{\alpha,\infty}$ are the Banach spaces constructed by $K_{\alpha,\theta}$-method of real interpolation between Banach spaces $B_0$ and $B_1$, with  $0<\alpha<1$ and $\theta=\infty$  \cite{BL80}; hereafter we need only the case $B_1\subset B_0$.
The spaces $H^{(\varkappa)}$ for noninteger $\varkappa$ are subspaces in the Nikolskii spaces $H_2^\varkappa(I)$, in more detail, see \cite{N69,BL80}.
Let $V(\bar{I})$ be the space of functions of bounded variation on $\bar{I}$. Then  $V(\bar{I})\subset H^{(1/2)}$ that allows one to cover the practically important cases of discontinuous piecewise differentiable functions for $\varkappa=\frac12$ and then functions having discontinuous $k$th order Sobolev derivative on $\bar{I}$  for $\varkappa=k+\frac12$, $k\geq 1$.

\par The exact explicit description of the spaces $F^{(\lambda)}$ is a separate problem in theory of functions of a real variable, but obviously they are more broad than when using the standard Sobolev spaces.
Indeed, consider the Sobolev subspace  $W_{(0)}^{5;2,1}(Q)$ endowed with the norm 
\[
 \|f\|_{W_{(0)}^{5;2,1}(Q)}=\sum_{k=0}^5\sum_{l=0}^{5-k}\|\partial_x^k\partial_t^lf\|_{L^{2,1}(Q)},
\]
where all the appearing derivatives belong to $L^{2,1}(Q)$ and additionally  $f|_{x=0,X}=\partial_x^2f|_{x=0,X}=\partial_x^2\partial_t^2f|_{x=0,X}=0$. Then $W_{(0)}^{5;2,1}(Q)\subset S_{2,1}^{3,2}W(Q)\cap S_{2,1}^{2,3}W(Q)$ and consequently 
\[
(L^{2,1}(Q),W_{(0)}^{5;2,1}(Q))_{\lambda/5,\infty}\subset F^{(\lambda)},\ \ 0<\lambda<5.
\]
Here we take into account the simple imbeddings
\[
S_{2,1}^{1,1}W(Q)\subset C(\bar{I}_T;W^{1,2}(I))\subset C(\bar{Q}),
\]
where $S_{2,1}^{1,1}W(Q)$ is the space of functions $f\in W^{1;2,1}(Q)$ having the dominating mixed derivative $\partial_x\partial_tf\in L^{2,1}(Q)$, endowed with the norm
\[
\|f\|_{S_{2,1}^{1,1}W(Q)}=
\|f\|_{L^{2,1}(Q)}
+\|\partial_xf\|_{L^{2,1}(Q)}
+\|\partial_tf\|_{L^{2,1}(Q)}
+\|\partial_x\partial_tf\|_{L^{2,1}(Q)}.
\]
Namely, these imbeddings guarantee that all the derivatives up to and including the 3rd order of the functions $f\in W_{(0)}^{5;2,1}(Q)$ are continuous in  $\bar{Q}$. 
Therefore, for such functions, we have  $f(x,t)|_{x=0,X}=\partial_x^2f(x,t)|_{x=0,X}=0$ for all $0\leq t\leq T$ and, consequently, the conditions of vanishing at $x=0,X$, including the corner points $(0,0),(X,0)\in\bar{Q}$, which participate  in the above definitions of the spaces $S_{2,1}^{3,2}W(Q)$ and $S_{2,1}^{2,3}W(Q)$, are valid. 

\par Let $w\in L^1(I)$ and $w(x)$ be extended oddly with respect to $x=0,X$ on $\tilde{I}$. We introduce the one more average 
\[
 q_{2h}w_i=q_hw_i-\tfrac{h^2}{12}\Lambda_xq_hw_i=\tfrac{1}{12}(-q_hw_{i-1}+14q_hw_i-q_hw_{i+1}),\ \ 1\leq i\leq N-1.
\]
Here $q_hw|_{i=0,N}=0$.
If $\partial_x^4w\in L^2(I)$ and $w|_{x=0,X}=\partial_x^2w|_{x=0,X}=0$,  the bound holds
\begin{gather}
 \|w-q_{2h}w\|_h\leq ch^4\|\partial_x^4w\|_{L^2(I)}.
\label{err of q2h}
\end{gather}
It follows from the formula
\[
q_{2h}w-w=(q_h-s_N)w-\tfrac{h^2}{12}\Lambda_x(q_hw-w)
\]
in virtue of the first bound \eqref{est qh min sN} and bound \eqref{est lam mim lam q}.

Now we are ready to state and prove the main result of this Section, i.e., the general error bound of the fractional order in terms of the data.
\begin{theorem}
\label{theo:general err bound}
Consider any $1\leq\lambda\leq 6$. Let $u_0\in H^{(\lambda)}$, $u_1\in  H^{(\lambda-1)}$ and $f\in F^{(\lambda-1)}$. 
Under the stability condition \eqref{stab cond 1}, 
for the compact scheme \eqref{eqht}--\eqref{bqht} with $v^0=u^0$, $u_{1\*h}=u_{1\*h}^{(2)}$ and $f_{\*h}=q_hq_\tau f$, the error bound holds 
\begin{gather}
\max_{1\leq m\leq M}\big(\|\bar{\delta}_t(q_{2h}u-v)^m\|_h+\|\bar{\delta}_x(u-v)^m\|_{h*}\big)
\nonumber\\
\leq ch^{4(\lambda-1)/5}\big(\|u_0\|_{ H^{(\lambda)}}+\|u_1\|_{ H^{(\lambda-1)}}+\|f\|_{F^{(\lambda-1)}}\big).
\label{gen err bound}
\end{gather}
\end{theorem}
\begin{proof}
First, for $\lambda=1$, we have
\begin{gather*}
\max_{1\leq m\leq M}\big(\|\bar{\delta}_t(q_{2h}u-v)^m\|_h+\|\bar{\delta}_x(u-v)^m\|_{h*}\big)
\\
\leq\max_{1\leq m\leq M}\big(\|\bar{\delta}_tq_{2h}u^m\|_h+\|\bar{\delta}_xu^m\|_{h*}\big)
+\max_{1\leq m\leq M}\big(\|\bar{\delta}_tv^m\|_h+\|\bar{\delta}_xv^m\|_{h*}\big)
\\
\leq \max_{0\leq t\leq T}\big(c\|\partial_tu(\cdot,t)\|_{L^2(I)}+\|\partial_xu(\cdot,t)\|_{L^2(I)}\big)
+4\max\Big\{\max_{1\leq m\leq M}\|\bar{\delta}_tv^m\|_B,\max_{1\leq m\leq M}\|\bar{\delta}_xv^m\|_{h*}\Big\},
\end{gather*}
where the norms of  $\bar{\delta}_tq_{2h}u$ and $\bar{\delta}_xu$ are bounded similarly as above using the inequalities $\|q_{2h}w\|_h\leq\frac43\|q_hw\|_h$ and \eqref{bound of fd of u}.
Therefore, in virtue of the energy bound \eqref{en bound u} and Corollary  \ref{corr: v en bound}, we get
\begin{gather*}
\max_{1\leq m\leq M}\big(\|\bar{\delta}_t(q_{2h}u-v)^m\|_h+\|\bar{\delta}_x(u-v)^m\|_{h*}\big)
\leq c\big(\|u_0\|_{H_0^1(I)}+\|u_1\|_{L^2(I)}+\|f\|_{L^{2,1}(Q)}\big).
\end{gather*}
This is the error bound \eqref{gen err bound} for  $\lambda=1$.

\par Second, let  $u_0\in\mathcal{H}^6(I)$ and $u_1\in\mathcal{H}^5(I)$.
In virtue of \cite[Proposition 1.3]{Z94},
for $f\in S_{2,1}^{3,2}W(Q)$, the following regularity property holds for the solution to the IBVP \eqref{eq}--\eqref{bq}:
\begin{gather} \|\partial_tu\|_{C(\bar{I}_T;\mathcal{H}^5(I))}
+\|\partial_t^2u\|_{C(\bar{I}_T;\mathcal{H}^4(I))}
+\|\partial_t^3u\|_{C(\bar{I}_T;\mathcal{H}^3(I))}
+\|\partial_t^4\partial_x^2u\|_{L^{2,1}(Q)}
\nonumber\\ 
\leq c\big(\|u_0\|_{\mathcal{H}^6(I)}+\|u_1\|_{\mathcal{H}^5(I)}+\|f\|_{S_{2,1}^{3,2}W(Q)}\big),
\label{reg bound 1}
\end{gather}
and, for $f\in S_{2,1}^{2,3}W(Q)$, another  regularity property holds
\begin{gather}
\|\partial_t^2u\|_{C(\bar{I}_T;\mathcal{H}^4(I))}
+\|\partial_t^3u\|_{C(\bar{I}_T;\mathcal{H}^3(I))}
+\|\partial_t^4u\|_{C(\bar{I}_T;\mathcal{H}^2(I))}
+\|\partial_t^5\partial_xu\|_{L^{2,1}(Q)}
\nonumber\\ 
\leq c\big(\|u_0\|_{\mathcal{H}^6(I)}+\|u_1\|_{\mathcal{H}^5(I)}+\|f\|_{S_{2,1}^{2,3}W(Q)}\big).
\label{reg bound 2}
\end{gather}
The derivatives of $u$ appearing in their left-hand sides exist and belong to the spaces in whose norms they stand.
Therefore, with the help of inequalities \eqref{lower bound en norm 1} and \eqref{lower bound en norm 2} for $\ve_1^2=\ve_0^2/3$, from the 4th order error bound \eqref{r bound 0} and Remark \ref{rem: another u1h}, for $f\in S_{2,1}^{3,2}W(Q)+S_{2,1}^{2,3}W(Q)$, 
the bound in terms of the data holds
\begin{gather}
\max_{1\leq m\leq M}\big(\|\bar{\delta}_t(u-v)^m\|_h+\|\bar{\delta}_x(u-v)^m\|_{h*}\big)
\nonumber\\
\leq
ch^4\big(\|u_0\|_{\mathcal{H}^6(I)}+\|u_1\|_{\mathcal{H}^5(I)}+\|f\|_{S_{2,1}^{3,2}W(Q)+S_{2,1}^{2,3}W(Q)}\big).
\label{second err bound order 4}
\end{gather}

\par Further, using bound \eqref{err of q2h} for $w-q_{2h}w$, we obtain the bound
\[
\max_{1\leq m\leq M}\|\bar{\delta}_t(u-q_{2h}u)\|_h\leq ch^4\|\partial_t\partial_x^4u\|_{L^{2,\infty}(Q)}.
\]
Here we use the properties $u(x,t)|_{x=0,X}=\partial_x^2u(x,t)|_{x=0,X}=0$ for any $0\leq t\leq T$ (notice that  $a^2\partial_x^2u=\partial_t^2u-f$, where $\partial_t^2u,f\in S_{2,1}^{1,1}W(Q)\subset C(\bar{Q})$).
Therefore, with the help of estimate \eqref{px4 pt1 bound} for $u$, bound \eqref{second err bound order 4} remains valid for $\bar{\delta}_t(u-v)$ replaced with $\bar{\delta}_t(q_{2h}u-v)$, i.e., the error bound  \eqref{gen err bound} holds for $\lambda=6$.

\par Finally, we consider the error operator $R$: $\*d=(u_0,u_1,f)\to R\*d= \{\bar{\delta}_t(q_{2h}u-v),\bar{\delta}_x(u-v)\}$;
note that $(u-v)^0=0$ on $\bar{\omega}_h$.
The appeared pair of errors belongs to the Banach space $B_0$ of pairs whose components $\{y,z\}$ are given on the meshes $\omega_h\times(\omega^\tau\cup\{T\})$ and $(\omega_h\cup\{X\})\times(\omega^\tau\cup\{T\})$, endowed with the norm $\|\{y,z\}\|_{B_0}=\max_{1\leq m\leq M}(\|y^m\|_h+\|z^m\|_{h*})$.

\par We apply the basic theorem on interpolation of linear operators \cite{BL80} to the operator $R$ using the proven error bound \eqref{gen err bound} for $\lambda=1$ and 6 which  expresses bounds for the corresponding norms of $R$. Since $R$ is linear, it is convenient to apply the theorem separately in $u_0$ (i.e., for $u_1=f=0$),  $u_1$ (i.e., for $u_0=f=0$) and $f$ (i.e., for $u_0=u_1=0$), then add the results and finally derive the error bound \eqref{gen err bound} for $1<\lambda<6$.
Here we applied the well-known imbedding   
$\mathcal{H}^{k+j}(I)\subset(\mathcal{H}^j(I),\mathcal{H}^{5+j}(I))_{k/5,\infty}$ for $k=1,2,3,4$ and $j=0,1$.
\end{proof}

Note that, in the error bound \eqref{gen err bound}, one can use  subspaces in the Nikolskii spaces $H_2^{\lambda}(I)$ and $H_2^{\lambda-1}(I)$ instead of the Sobolev subspaces $H^{(\lambda)}$ and $H^{(\lambda-1)}$ for integer $\lambda=2,3,4,5$ as well (the former subspaces are slightly broader than the latter ones), but this is less explicit and has few applications for the wave equation.  

\section{\normalsize Lower error bounds}
\label{sec:lower bounds}
\setcounter{equation}{0}
\setcounter{lemma}{0}
\setcounter{theorem}{0}
\setcounter{remark}{0}
\setcounter{corollary}{0}

\par To simplify further formulas, we can assume that
$X=\pi$ and $a=1$  by scaling.
Recall the well-known spectral formulas 
\begin{equation}
-\Lambda_x\sin kx=\lambda_k\sin kx\ \ \text{on}\ \ \omega_h,\ \ 
\lambda_k=\Big(\dfrac{2}{h}\sin\dfrac{kh}{2}\Big)^2,\ \ 1\leq k\leq N-1.
\label{eig values}
\end{equation}

\par For $\*d(x,t)=(\alpha_0,\alpha_1,g(t))\sin kx$, where $\alpha_0$ and $\alpha_1$ are constants and
$g\in L^1(I_T)$,
the classical Fourier formula represents the solution to the IBVP \eqref{eq}--\eqref{bq}:
\begin{equation}
\label{4.1}
u(x,t)=\Big(\alpha_0\cos kt+\dfrac{\alpha_1}{k}\sin kt+\dfrac{1}{k}\int_0^t g(\theta)\sin k(t-\theta)\,d\theta\Big)\sin kx\ \ \text{on}\ \ \bar{Q}.
\end{equation}

\par We need its counterpart for the compact scheme. 
For $y\in C(\bar{I}_T)$, let $\hat{s}_t y$ be its piecewise linear interpolant, i.e.,
$\hat{s}_ty(t_m)=y(t_m)$ on $\bar{\omega}^\tau$ and $\hat{s}_ty$ is linear on the segments $[t_{m-1},t_m]$, $1\leq m\leq M$.
\begin{lemma}
\label{lemma41}
Let the stability condition \eqref{stab cond 1} be valid and $1\leq k\leq N-1$. For $\*d(x,t)=(\alpha_0,\alpha_1,g(t))\sin kx$ with $g\in L^1(I_T)$, the solution to the compact scheme \eqref{eqht}--\eqref{bqht} with $v^0=u_0$, some $u_{1\*h}$ such that $u_{1\*h}=a_{1k}\alpha_1\sin kx$ on $\bar{\omega}_h$ and $f_{\*h}=q_hq_tf$, is represented as
\begin{equation}
\label{4.2}
v(x,t)=\Big(
\alpha_0\cos \mu_k t+\hat{\gamma}_{1k}\dfrac{\alpha_1}{k}\sin \mu_k t 
+ \frac{\gamma_{1k}}{k}\int_0^t g(\theta)\hat{s}_\theta\sin \mu_k(t-\theta)\,d\theta\Big)\sin kx\ \ \text{on}\ \ \bar{\omega}_{\*h},
\end{equation}
where the coefficients $\mu_k$, $\hat{\gamma}_{1k}$ and $\gamma_{1k}$ are given by the formulas
\begin{gather}
\mu_k=\dfrac{2}{\tau}\arcsin\dfrac{\tau\varphi_k}{2},\ \ \varphi_k=\Big(\dfrac{\lambda_k}{1-(h^2/6)\lambda_k
+\tau^2\sigma_N\lambda_k}\Big)^\frac{1}{2},
\label{4.3}\\
\hat{\gamma}_{1k}=a_{1k}\frac{2k}{\lambda_k\tau}\tan \frac{\mu_k\tau}{2},\ \ \gamma_{1k}=\frac{2}{k\tau}
\tan\frac{\mu_k\tau}{2}.
\label{4.5}
\end{gather}

\par In particular, for $u_{1\*h}=u_{1\*h}^{(0)}$, $u_{1\*h}^{(1)}$ and $u_{1\*h}^{(2)}$, we respectively have
\begin{gather}
a_{1k}=1-\frac{h^2+\tau^2}{12}\lambda_k,\ \ 
\frac{\lambda_k}{k^2}\Big(1-\frac{\tau^2k^2}{12}\Big),\ \
\frac{\lambda_k}{k^2}\Big(1-\frac{\tau^2\lambda_k}{12}\Big).
\label{a1 for u1h}
\end{gather}
Hereafter, for brevity, the dependence of the coefficients on $\*h$ is not indicated.
\end{lemma}

\par This lemma follows from the more general one \cite[Lemma 1.1]{Z25} since the compact scheme can be considered as the bilinear FEM with the special weight $\sigma=\sigma_N$ that has already been used above. 

\par Notice that $1-\frac{h^2}{6}\lambda_k
+\tau^2\sigma_N\lambda_k=1+\frac{\tau^2-h^2}{12}\lambda_k$.
Due to the stability condition  \eqref{stab cond 1}, it is not difficult to check that
\[
\min\Big\{\frac23,(1+\ve_1)\frac{\tau^2}{4}\lambda_k\Big\}
\leq 1 +\frac{\tau^2-h^2}{12}\lambda_k\leq 1\ \ \text{with}\ \ \ve_1=\frac23\frac{\ve_0^2}{1-\ve_0^2/2}
\]
and thus
\begin{gather}
 \frac{\tau}{2}\sqrt{\lambda_k}\leq\frac{\tau\vp_k}{2}\leq \min\Big\{\frac{1}{\sqrt{1+\ve_1}}, \sqrt{\frac{3}{2}}\frac{\tau}{2}\sqrt{\lambda_k}\Big\}.
\label{fi k bound}
\end{gather}

Let us expand the terms $\mu_k$, $\hat{\gamma}_{1k}$ and $\gamma_{1k}$. 
\begin{lemma}
\label{lemma52}
Let the stability condition \eqref{stab cond 1} be valid and $1\leq k\leq N-1$. The asymptotic formulas hold: $\mu_k \asymp k$, i.e., $\underline{c}k\leq\mu_k\leq \bar{c}k$ with some $0<\underline{c}<\bar{c}$ independent of $\*h$ and $k$, and 
\begin{gather}
\mu_k = k-k^5\nu_{\*h}+{\mathcal O}(k^7h^6)\ \ \text{with}\ \ \nu_{\*h}:=\tfrac{1}{480}(h^4-\tau^4),\ \ \tfrac{\ve_0^2}{2} h^4\leq 480\nu_{\*h}\leq h^4,
\label{4.10}\\[1mm]
\hat{\gamma}_{1k}=1+{\mathcal O}\big((kh)^2\big),\ \
\gamma_{1k}=1+{\mathcal O}\big((kh)^2\big)
\label{4.11}
\end{gather}
for any $a_{1k}=1+{\mathcal O}\big((kh)^2\big)$ including those given by formulas \eqref{a1 for u1h}.
\end{lemma}
\begin{proof}
The relation $\mu_k \asymp k$ together with 
$\mu_k\frac{\tau}{2}\leq\frac{\pi}{2}\frac{\tau\varphi_k}{2}\leq\frac{\pi}{2\sqrt{1+\ve_1}}$
follow from inequalities \eqref{fi k bound} and the similar elementary relation $\sqrt{\lambda_k} \asymp k$.

\par We set $\hat{h}=\frac{h}{2}$ and $\hat{\tau}=\frac{\tau}{2}$. Clearly the expansion holds
\[
\sqrt{\lambda_k}=k\Big[1-\frac{1}{6}(k\hat{h})^2+\frac{1}{120}(k\hat{h})^4+{\mathcal O}\big((kh)^6\big)\Big]
\]
and therefore
$\frac{\lambda_k}{k^2}=1-\frac13(k\hat{h})^2
+{\mathcal O}\big((kh)^4\big)$ and
$\frac{\lambda_k^2}{k^4}=1+{\mathcal O}\big((kh)^2\big)$.
We also set $a:=\frac13(\frac{\hat{\tau}}{\hat{h}}-1)$. Then we can write
\begin{gather*}
\frac{\varphi_k}{k}=\Big(\frac{\lambda_k}{1+a\hat{h}^2\lambda_k}\Big)^{1/2}
=\frac{\sqrt{\lambda_k}}{k}\Big[1-\frac12a\frac{\lambda_k}{k^2}(k\hat{h})^2+\frac38a^2\frac{\lambda_k^2}{k^4}(k\hat{h})^4+{\mathcal O}\big((kh)^6\big)\Big]
\\
=\Big[1-\frac{1}{6}(k\hat{h})^2+\frac{1}{120}(k\hat{h})^4\Big]\Big[1-\frac{a}{2}(k\hat{h})^2+\Big(\frac{a}{6}+\frac38a^2\Big)(k\hat{h})^4\Big]+{\mathcal O}\big((kh)^6\big)
\\
=1-\Big(\frac16+\frac{a}{2}\Big)(k\hat{h})^2+\Big(\frac{1}{120}+\frac{a}{4}+\frac38 a^2\Big)(k\hat{h})^4+{\mathcal O}\big((kh)^6\big).
\end{gather*}
Thus we obtain the expansion
\begin{equation*}
\frac{\varphi_k}{k}=1-\frac{1}{6}(k\hat{\tau})^2+b(k\hat{h})^4+{\mathcal O}\big((kh)^6\big)\ \  \text{with}\ \ b:=\frac{1}{24}\frac{\tau^2}{h^2}-\frac{1}{30}.
\end{equation*}

\par Based on this expansion, further we derive
\begin{gather*}
\frac{\mu_k}{k}=\frac{1}{k\hat{\tau}}\arcsin\hat{\tau}\varphi_k=
\frac{\varphi_k}{k}+\frac16\Big(\frac{\varphi_k}{k}\Big)^3(k\hat{\tau})^2+
\frac{3}{40}\Big(\frac{\varphi_k}{k}\Big)^5(k\hat{\tau})^4+{\mathcal O}\big((k\hat{\tau})^6\big)
\nonumber\\
=1-\frac{1}{6}(k\hat{\tau})^2+b(k\hat{h})^4
+\frac16\big((k\hat{\tau})^2-3\frac16(k\hat{\tau})^4\big)+\frac{3}{40}(k\hat{\tau})^4
+{\mathcal O}\big((kh)^6\big)
\nonumber\\
=1+k^4\Big(\frac{1}{24}\hat{\tau}^4-\frac{1}{30}\hat{h}^4-\frac{1}{120}\hat{\tau}^4\Big)+{\mathcal O}\big((kh)^6\big)
=1+k^4\frac{1}{30}(\hat{\tau}^4-\hat{h}^4)+{\mathcal O}\big((kh)^6\big),
\end{gather*} 
and expansion \eqref{4.10} is proved.

\par Expansions \eqref{4.11} and $a_{1k}=1+{\mathcal O}\big((kh)^2\big)$ for $a_{1k}$ given by formulas \eqref{a1 for u1h} follow from formulas $\lambda_k=k^2(1+{\mathcal O}\big((kh)^2\big)$,  $\mu_k=k(1+{\mathcal O}\big((kh)^2\big)$ and the above inequality $\mu_k\frac{\tau}{2}\leq\frac{\pi}{2\sqrt{1+\ve_1}}$.
\end{proof}
\begin{corollary}
\label{coroll:2.1}
Let the stability condition \eqref{stab cond 1} be valid.
For any $\alpha>0$ independent of $\*h$, there exists an integer $1\leq k_{\*h}\leq N-1$ such that $k_{\*h}\asymp h^{-4/5}$ and
\[
\mu_{k_{\*h}}=k_{\*h}-\alpha+
{\mathcal O}\big(h^{2/5}\big).
\]
\end{corollary}
\begin{proof}
We define 
$k_{\*h}=[\rho_{\*h}]+1$, where $\rho_{\*h}=\big(\frac{\alpha}{\nu_{\*h}}\big)^{1/5}$ and $[\rho_{\*h}]$ is its integer part.
For $h$ so small that
$(\rho_{\*h}+2)h\leq \pi$ 
(without loss of generality), we get
$1\leq k_{\*h}\leq N-1$ and
$0\leq k_{\*h}-\rho_{\*h}\leq 1$.  
Using relations \eqref{4.10}, we obtain $k_{\*h}\asymp\rho_{\*h}\asymp h^{-4/5}$ and
\[
\mu_{k_{\*h}}-(k_{\*h}-\alpha)=(\rho_{\*h}^5-k_{\*h}^5)\nu_{\*h}+{\mathcal O}(k_{\*h}^7h^6)
={\mathcal O}(\rho_{\*h}^4\nu_{\*h})+{\mathcal O}(k_{\*h}^7h^6)
={\mathcal O}\big(h^{2/5}\big)
\]
that proves the result.
\end{proof}

\par Let $x_{i-1/2}=(i-0.5)h$ and $w_{i-1/2}=w(x_{i-1/2})$, $1\leq i\leq N$. We define the mesh $L^1$ norms
\begin{gather*}
\|w\|_{L_h^1(I_\pi)}=\sum_{i=1}^N\tfrac12(|w_{i-1}|+|w_i|)h,\ \ \|w\|_{L_{h*}^1(I_\pi)}=\sum_{i=1}^N|w_{i-1/2}|h,\ \ \|\bar{\delta}_xw\|_{L_h^1(I_\pi)}=\sum_{i=1}^N|\delta_xw_i|h,
\\ 
\|y\|_{L_\tau^1(I_T)}=\sum_{j=1}^M\tfrac12(|y_{j-1}|+|y_j|)\tau,\ \ \|v\|_{L_{\*h}^1(Q_T)}=\|\|v\|_{L_h^1(I_\pi)}\|_{L_\tau^1(I_T)}.
\end{gather*}
For $w\in W^{1,1}(I_\pi)$ and $y\in W^{1,1}(I_T)$, the simple bounds for differences between the continuous and mesh $L^1$ norms hold
\begin{gather}
\max\{\big|\|w\|_{L^1(I_\pi)}-\|w\|_{L_h^1(I_\pi)}\big|,   \big|\|w\|_{L^1(I_\pi)}-\|w\|_{L_{h*}^1(I_\pi)}\big|
\leq \|w'\|_{L^1(I_\pi)}h,
\label{bound for diff of norms}
\\
\big|\|y\|_{L^1(I_T)}-\|y\|_{L_\tau^1(I_T)}\big|
\leq \|y'\|_{L^1(I_T)}\tau.
\label{bound for diff of norms 2}
\end{gather}

\par For any $a>0$ and $y\in W^{1,1}(I_T)$, the estimate holds 
\begin{equation}
\Big|\int_{0}^T |y(t)\sin at|\,dt-\frac{2}{\pi}\int_{0}^T |y(t)|\,dt\Big| \leq \Big(\|y'\|_{L^1(I_T)}+
\frac32\big(1+\frac{\pi}{2}\big)\|y\|_{L^\infty(I_T)}\Big)\frac{2}{a},
\label{4.14}
\end{equation}
and it remains valid for $\sin at$ replaced with $\cos at$ (that is used below), see \cite{Z25}.

\par Below we use the collections of the harmonic data 
\[
\*d_k^{(0)}=(\sin kx,0,0),\ 
\*d_k^{(1)}=(0,\sin kx,0),\ 
\*d_k^{(2)}=(0,0,(\sin kx)\sin(k-1)t).
\]
Let $(u-v_{\*h})[\*d]$ be the compact scheme error for a given $\*d$.

\par The proof of the lower error bounds is based on the following asymptotic behavior of the error norms for these specific data.
\begin{theorem}
\label{theo:4.1} 
Let $l=0,1$ and $j=0,1,2$. 
For some $k=k_{\*h}\asymp h^{-4/5}$, the following formula for the norms of the compact scheme error holds
\begin{gather}
\|\bar{\delta}_x^l(u-v_{\*h})[\*d_k^{(j)}]\|_{L_{\*h}^1(Q_T)}=k^{-p_j+l}\Big(\dfrac{4}{\pi}c_j(T)+{\mathcal O}(h^{1/5})\Big),
\label{4.15b}
\end{gather}
with $p_0=0$ and $p_1=p_2=1$.
Also $c_0(T)=c_1(T)=2(2K_T+1-\cos(T-K_T\pi))$, where $K_T$ is the integer part of $\frac{T}{\pi}$, and $c_2(T)=T-\sin T$.
\end{theorem}
\begin{proof}
Let $l=0,1$ and
$(x,t)\in\bar{\omega}_{\*h}$ excluding $x=0$ for $l=1$. We use $\alpha=2$ in Corollary \ref{coroll:2.1}.
\par 1. First we consider the cases $j=0,1$.
Due to the Fourier-type representations of solutions  \eqref{4.1}--\eqref{4.2}, we obtain the formulas for the error
\begin{gather}
\bar{\delta}_x^l(u-v_{\*h})[\*d_k^{(0)}](x,t)=
[\cos kt-\cos(k-2)t+r_k^{(0)}(t)]\,\bar{\delta}_x^l\sin kx,
\label{4.16}\\
\bar{\delta}_x^l(u-v_{\*h})[\*d_k^{(1)}](x,t)=\dfrac{1}{k}[\sin kt-\sin(k-2)t+r_k^{(1)}(t)]\,\bar{\delta}_x^l\sin kx.
\label{4.16_1}
\end{gather}
Here clearly we have
\begin{gather}
\cos kt-\cos(k-2)t=-2(\sin t)\sin(k-1)t,\ \
\sin kt-\sin(k-2)t=2(\sin t)\cos(k-1)t.
\label{trig form 1}
\end{gather}
The remainders are
\[
r_k^{(0)}(t)=\cos(k-2)t-\cos\mu_kt,\ \
r_k^{(1)}(t)=\sin(k-2)t-\sin\mu_kt
-(\hat{\gamma}_{1k}-1)\sin \mu_k t,
\]
and they satisfy the bounds
\begin{gather*}
\|r_k^{(0)}\|_{C(\bar{I}_T)}\leq |\mu_k-(k-2)|,\ \ 
\|r_k^{(1)}\|_{C(\bar{I}_T)}\leq |\mu_k-(k-2)|+|\hat{\gamma}_{1k}-1|.
\end{gather*}

\par Moreover, due to estimates \eqref{bound for diff of norms}, \eqref{bound for diff of norms 2} and \eqref{4.14}, we obtain
\begin{gather}
\|\sin kx\|_{L_h^1(I_\pi)}
=\|\sin kx\|_{L^1(I_\pi)}+{\mathcal O}(kh)
=2+{\mathcal O}(kh),
\label{sin mesh norm}\\
\|\bar{\delta}_x\sin kx\|_{L_h^1(I_\pi)}
=\sqrt{\lambda_k}\|\cos kx\|_{L_{h*}^1(I_\pi)}
=k(\|\cos kx\|_{L^1(I_\pi)}+{\mathcal O}(kh))
=k(2+{\mathcal O}(kh)),
\label{diff sin mesh norm}\\
\|2(\sin t)\sin (k-1)t\|_{L_\tau^1(I_T)}
=\|2(\sin t)\sin (k-1)t\|_{L^1(I_T)}+{\mathcal O}(k\tau)
\nonumber\\
=\frac{4}{\pi}\|\sin t\|_{L^1(I_T)}+{\mathcal O}(k\tau+k^{-1})
=\frac{2}{\pi}c_0(T)+{\mathcal O}(k\tau+k^{-1})\ \ \text{for}\ \ k\geq 2,
\label{mult sin mesh norm}
\end{gather}
and the last formula remains valid for $\sin (k-1)t$ replaced with $\cos (k-1)t$.

\smallskip\par 2. To consider the last case $j=2$, we define the functions
\begin{gather}
y^{(\vk)}(t)
\nonumber\\
:=\int_0^t (\sin(k-1) \theta) \sin \vk(t-\theta)\,d\theta=-\dfrac{1}{2}\dfrac{\sin(k-1)t-\sin\vk t}{(k-1)-\vk}+\dfrac{1}{2}\dfrac{\sin(k-1)t+\sin \vk t}{(k-1)+\vk},
\label{ya of t}\\
y_\vk(t):=\int_0^t (\sin (k-1)\theta)\hat{s}_\theta\sin \vk(t-\theta)\,d\theta,
\nonumber
\end{gather}
with the parameter $\vk$ such that $|\vk|\neq k-1$.

\par Once again due to the Fourier-type representations of solutions  \eqref{4.1}--\eqref{4.2}, we can write the formula for the error
\begin{equation}
\label{4.18}
\bar{\delta}_x^l(u-v_{\*h})[\*d_k^{(2)}](x,t)=\frac{1}{k}\big[(y^{(k)}-y^{(k-2)})(t)+r_k^{(2)}(t)\big]\,\bar{\delta}_x^l\sin kx.
\end{equation}
Here according to formula \eqref{ya of t}, we get
\begin{gather*}
 (y^{(k)}-y^{(k-2)})(t)=-\frac12[\sin kt-2\sin(k-1)t+\sin(k-2)t]+\frac{\theta_k}{k}=(1-\cos t)\sin(k-1)t+\frac{\theta_k}{k}
\end{gather*}
with some $\theta_k\in [-1,1]$.
The reminder is
\[
r_k^{(2)}(t):=(y^{(k-2)}-y^{(\mu_k)})(t)+(y^{(\mu_k)}-y_{\mu_k})(t)-(\gamma_{1k}-1)y_{\mu_k}(t),
\]
and it obeys the bound
\begin{gather*}
\|r_k^{(2)}\|_{C(\bar{I}_T)}\leq T|\mu_k-(k-2)|
+\max_{0\leq t\leq T}\int_0^t|\sin\mu_k(t-\theta)-\hat{s}_\theta\sin\mu_k(t-\theta)|\,d\theta+T|\gamma_{1k}-1|
\\
\leq T\big(|\mu_k-(k-2)|+\mu_k^2\tau^2+|\gamma_{1k}-1|\big),
\end{gather*}
where the elementary bound for the error of the linear interpolation has been applied.

\smallskip\par 3. Now we choose $k=k_{\*h}\asymp h^{-4/5}$ according to Corollary \ref{coroll:2.1} for $\alpha=2$. Then, for $j=0,2$, using the above bounds for the reminders, expansions \eqref {4.11} and Corollary \ref{coroll:2.1}, we obtain
\begin{equation*}
\big|\bar{\delta}_x^l(u-v_{\*h})[\*d_{k_{\*h}}^{(j)}](x,t)\big|=k_{\*h}^{-p_j}\big(|\zeta_j(t)\sin (k_{\*h}-1)t||\bar{\delta}_x^l\sin k_{\*h}x|+{\mathcal O}(h^{2/5})\big)
\end{equation*}
where $\zeta_0(t)=2\sin t$, $\zeta_2(t)=1-\cos t$ and ${\mathcal O}$-term is independent of $(x,t)$.
For $j=1$, the same formula with $\sin (k_{\*h}-1)t$ replaced with $\cos (k_{\*h}-1)t$ is valid.

\par To derive formula \eqref{4.15b}, it remains to apply formulas \eqref{sin mesh norm}--\eqref{mult sin mesh norm} and the formula similar to \eqref{mult sin mesh norm} with $2\sin t$ and $c_0(T)$ replaced with $1-\cos t$ and $c_1(T)$.
\end{proof}

\par Now we are ready to pass to the lower error bounds. For real $\lambda\geq 0$, let $C^{\lambda}(\bar{I}_\pi)$ and $C^{\lambda}(\bar{Q}_T)$ be the H\"{o}lder spaces of functions defined on $\bar{I}_\pi$ and $\bar{Q}_T$, for example, see \cite{LU}.
Recall that, for integer $\ell=\lambda$, they consist of functions continuous, for $\ell=0$, and $\ell$ times continuously differentiable, for $\ell\geq 1$, in $\bar{I}_\pi$ and $\bar{Q}_T$, respectively.
Denote by $C_{(0)}^{\lambda}(\bar{I}_\pi)$ and $C_{(0)}^{\lambda}(\bar{Q}_T)$ their subspaces (equipped with the same norms) containing functions such that $\partial_x^{2k}w(x)|_{x=0,\pi}=0$ and 
$\partial_x^{2k}f(x,t)|_{x=0,\pi}=0$ together with $\partial_t^{2k}f(x,t)|_{t=0}=0$, respectively, for $0\leq 2k\leq[\lambda]$;
note that here only the derivatives of the even order are involved. 
\begin{theorem}
\label{theo: low err est f}
Let the stability condition \eqref{stab cond 1} hold and $l=0,1$. There exist $h_0>0$ and $c_1>0$ such that, for $h\leq h_0$, the following lower error bounds with respect to 
$u_0$, $u_1$ and $f$ hold
\begin{gather}
\sup_{u_j(x)=\sin kx,\, k\in\mathbb{N}}\frac{\|\bar{\delta}_x^l(s_xu-v_{\*h})[\*d^{(j)}]\|_{L_{\*h}^1(Q_T)}}
{\|u_j\|_{C_{(0)}^{\lambda-j}(\bar{I}_\pi)}}\geq c_1h^{4(\lambda-l)/5},\ \ j=0,1,
\label{low err 1}\\ 
\sup_{f(x,t)=(\sin kx)\sin (k-1)t,\, k\in\mathbb{N},\,k\geq 2}\frac{\|\bar{\delta}_x^l(s_xu-v_{\*h})[\*d^{(2)}]\|_{L_{\*h}^1(Q_T)}}
{\|f\|_{C_{(0)}^{\lambda-1}(\bar{Q}_T)}}\geq c_1h^{4(\lambda-l)/5},
\label{low err 3}
\end{gather}
for any $l\leq\lambda\leq 5+l$ such that, for $l=0$, $\lambda\geq 1$  in \eqref{low err 1} for $j=1$ and  in 
\eqref{low err 3}.
\end{theorem}
\begin{proof}
For natural $k$, it is not difficult to check (for example, see \cite{Z25}) that the bounds hold
\begin{gather}
 \|\sin kx\|_{C_{(0)}^{\lambda}(\bar{I}_\pi)}\leq ck^\lambda,\ \ 
 \|(\sin kx)\sin (k-1)t\|_{C_{(0)}^{\lambda}(\bar{Q}_T)}\leq ck^\lambda\ \ \text{for}\ \ \lambda\geq 0.
\label{bounds for sin}
\end{gather}

\par Due to Theorem \ref{theo:4.1} and for  $k=k_{\*h}$ chosen in it, the lower bound holds
\begin{gather*}
\|\bar{\delta}_x^l(s_xu-v_{\*h})[\*d_{k_{\*h}}^{(j)}]\|_{L^1(Q_T)}\geq c_1k_{\*h}^{-p_j+l},\ \ j=0,1,2, \end{gather*}
for $h\leq h_0$ and some $c_1>0$, with $h_0>0$ small enough.
Combining it with bound \eqref{bounds for sin}, we get
\begin{gather*}
\frac{\|\bar{\delta}_x^l(s_xu-v_{\*h})[\*d_{k_{\*h}}^{(j)}]\|_{L_{\*h}^1(Q_T)}}
{\|\sin k_{\*h}x\|_{C_{(0)}^{\lambda-j}(\bar{I}_\pi)}}\geq\frac{c_1}{k_{\*h}^{\lambda-l}},\ \ j=0,1,  
\ \
\frac{\|\bar{\delta}_x^l(s_xu-v_{\*h})[\*d_{k_{\*h}}^{(2)}]\|_{L_{\*h}^1(Q_T)}}
{\|(\sin k_{\*h}x)\sin(k_{\*h}-1)t\|_{C_{(0)}^{\lambda-1}(\bar{Q}_T)}}\geq\frac{c_1}{k_{\*h}^{\lambda-l}} 
\end{gather*}
for any $l\leq\lambda\leq 5+l$ such that, for $l=0$, we assume that  $\lambda\geq 1$ in the former bound with $j=1$ and in the latter one.
Due to $k_{\*h}\asymp h^{-4/5}$, these bounds lead to bounds \eqref{low err 1}--\eqref{low err 3}.
\end{proof}

\par Note that actually bounds \eqref{low err 1}--\eqref{low err 3} are valid for any $l\leq\lambda\leq 5+l$ without exceptions, provided that the H\"{o}lder-type spaces of distributions $C_{(0)}^{\lambda}(\bar{I}_\pi)$ and $C_{(0)}^{\lambda}(\bar{Q}_T)$ for $-1\leq\lambda<0$ are suitably defined, see \cite{Z25}, but we do not dwell on that in the present paper.

\par Finally, the following imbeddings hold
\[
\|w\|_{H^{(\lambda)}}\leq c\|w\|_{C_{(0)}^{\lambda}(\bar{I}_\pi)},\ 0\leq\lambda\leq 6;\ \ \|f\|_{F^{\lambda}}\leq 
 c\|f\|_{C_{(0)}^{\lambda}(\bar{Q}_T)},\ 0\leq\lambda\leq 5,
\]
for any $w\in C_{(0)}^{\lambda}(\bar{I}_\pi)$, $f\in C_{(0)}^{\lambda}(\bar{Q}_T)$ and $X=\pi$.
The latter one follows from the elementary property that if the four Banach spaces $B_1\subset B_0$ and $\hat{B}_1\subset \hat{B}_0$ are such that $\hat{B}_i\subset B_i$, $i=0,1$, then $(\hat{B}_0,\hat{B}_1)_{\alpha,\infty}\subset(B_0,B_1)_{\alpha,\infty}$ for $0<\alpha<1$.
According to the lower bounds \eqref{low err 1}--\eqref{low err 3} for $l=1$, the error bound \eqref{gen err bound} is sharp in order for each $0\leq\lambda\leq 5$ with respect to each of the functions $u_0$, $u_1$ and $f$. Moreover, the error bound cannot be improved if the summability index in the error norm is weakened down to 1 both with respect to $x$ and $t$ and simultaneously the summability index in the norms of data is strengthened up to $\infty$ both with respect to $x$ and $t$. In addition, for $f$, passing from the dominating mixed smoothness to the standard one cannot improve the error orders as well. 

\section*{\normalsize Acknowledgments}

This study was supported by the RSF, project 23-21-00061 (Sections \ref{sec: ibvp and scheme}--\ref{sec:fract order bounds}) and the Basic Research Program at the HSE University (Section \ref{sec:lower bounds}).

\makeatletter
\renewcommand{\@biblabel}[1]{#1.\hfill}
\makeatother
\renewcommand{\refname}{\centerline{\normalsize\rm \textbf{Reference}s}}
\end{document}